\documentclass[11pt,a4paper,reqno]{amsart}

\usepackage{tikz}
\usetikzlibrary{decorations.markings}
\usetikzlibrary{arrows,matrix}
\usepgflibrary{arrows}

\usepackage{clrscode}
\usepackage{mathrsfs,amssymb}

\usepackage{amsmath, amsthm}

\newcommand{\topp}[1]{\left\lceil{#1}\right\rceil}
\newcommand{\bott}[1]{\left\lfloor{#1}\right\rfloor}

\newtheorem{theorem}{Theorem}[section]
\newtheorem*{theorema}{Theorem A}
\newtheorem*{theoremb}{Theorem B}
\newtheorem{lemma}[theorem]{Lemma}
\newtheorem{corollary}[theorem]{Corollary}

\newtheorem{proposition}[theorem]{Proposition}

\newcommand{\im}{\mathrm{im}}

\newcommand{\rinf}{\mathbb{R}^\infty}

\newcommand{\gr}{\mathcal{R}}

\newcommand{\gh}{\mathcal{H}}

\begin{document}

\title[Quasi-isometry and finite presentation]{Quasi-isometry and finite presentations of left cancellative monoids.}

\maketitle

\begin{center}

    ROBERT D. GRAY\footnote{
Centro de \'{A}lgebra da Universidade de Lisboa, Av.~Prof.~Gama Pinto, 2, 1649--003 Lisboa, Portugal.
Email \texttt{rdgray@fc.ul.pt}. Research supported by
FCT and FEDER, project POCTI-ISFL-1-143 of Centro de \'{A}lgebra da
Universidade de Lisboa, and by the project PTDC/MAT/69514/2006. He
gratefully acknowledges the support of EPSRC grant EP/F014945/1 and the
hospitality of the University of Manchester during a visit to Manchester.} \ and \ MARK KAMBITES\footnote{School of Mathematics, University of Manchester, Manchester M13 9PL, England. Email \texttt{Mark.Kambites@manchester.ac.uk}.
Research supported by an RCUK Academic Fellowship and
by EPSRC grant EP/F014945/1.} \\
\end{center}

\begin{abstract}
We show that being finitely presentable, and being finitely presentable with
solvable word problem are quasi-isometry invariants of finitely generated
left cancellative monoids. Our main tool is an elementary, but useful,
geometric characterisation of finite presentability for left cancellative
monoids. We also give examples to show that this characterisation does not
extend to monoids in general, and indeed that properties such as solvable
word problem are not isometry invariants for general monoids.
\end{abstract}

\section{Introduction}\label{sec_intro}

A focus of much recent research has been the extent to which geometric methods
developed in group theory can be applied to wider classes of monoids and
semigroups. A key concept in geometric group theory is that of
\textit{quasi-isometry}: an equivalence relation on the class of metric
spaces which captures their ``large-scale'' geometry.

In recent papers \cite{K_semimetric,K_svarc}, we have introduced notions of
quasi-isometry for \textit{semimetric} spaces (spaces equipped with asymmetric, partially
defined distance functions), and hence for monoids. We proved a semigroup-theoretic
analogues of the \textit{\v{S}varc-Milnor lemma}, showing that a monoid acting
in a suitably controlled way by isometric embeddings on a semimetric space must be
quasi-isometric to that space (see \cite[Theorem~4.1]{K_svarc}
for a precise statement).

One of the main reasons quasi-isometry is important in geometric group
theory, is that the quasi-isometry type of a group is a geometric invariant
which encapsulates many important algebraic and combinatorial properties of
the group.
The aim of this note is to show that this is also true for left cancellative
monoids. For example, just as for groups, the existence of a finite
presentation is a quasi-isometry invariant of left cancellative, finitely generated monoids:
\begin{theorema}
Let $M$ and $N$ be left cancellative, finitely generated monoids which are
quasi-isometric. Then $M$ is finitely presentable if and only if $N$ is
finitely presentable.
\end{theorema}

Similarly, it is known that solvability of the word problem is a quasi-isometry
invariant of finitely presented groups \cite{Alonso90} (although it remains
open if it is a quasi-isometry invariant of more general finitely generated
groups \cite[Section 3.7]{Bowditch06}).
It transpires that the same
holds for left cancellative monoids.

\begin{theoremb}
Let $M$ and $N$ be left cancellative, finitely presentable monoids which
are quasi-isometric. Then $M$ has solvable word problem if and only if $N$
has solvable word problem.
\end{theoremb}

Our proofs, which are given in Section~\ref{sec_fp} below,
are in spirit similar to those known in the group case. They are not, however, 
entirely straightforward generalisations, since much of the standard geometric
machinery used in the group case must be replaced with ``directed'' analogues,
the theory of which is less well developed.

Our proof methods do not readily generalise to non-left-cancellative monoids, which seem
to be more fundamentally ``non-geometric'' objects, in the sense that
relatively little of their structure can be discerned even from their
exact Cayley graphs, let alone from their quasi-isometry types. To
illustrate this, in Section~\ref{sec_examples} we exhibit an uncountable
family of finitely generated monoids which are pairwise non-isomorphic and
differ in important respects (such as for example solvability of the word
problem), but which share exactly the same unlabelled Cayley graph. It
remains an open question whether finite presentability is even an
isometry invariant, let alone a quasi-isometry invariant, for finitely
generated monoids in general.

\section{Preliminaries}\label{sec_prelim}

In this section, we briefly recall some basic definitions which are
essential for considering monoids as geometric objects.

Let $\rinf$ denote the set
$\mathbb{R}^{\geq 0} \cup \lbrace \infty \rbrace$ of
non-negative real numbers with $\infty$ adjoined. We equip it with the
obvious order, addition and multiplication, leaving $0 \infty$ undefined.
Now let $X$ be a set. A function $d : X \times X \to \rinf$ is called a
\emph{semimetric} on $X$ if:
\begin{itemize}
\item[(i)] $d(x,y)=0$ if and only if $x=y$; and
\item[(ii)] $d(x,z) \leq d(x,y) + d(y,z)$;
\end{itemize}
for all $x,y,z \in X$. A set equipped with a semimetric on it is a
\textit{semimetric space}. A useful example of a semimetric space is a
directed graph, with the distance between two vertices defined to be
the length of the shortest directed path between them, or $\infty$ if
there is no such path.

Now let $f : X \to Y$ be a map between semimetric spaces $X$ and $Y$. Write
$d_X$ and $d_Y$ for the semimetrics on $X$ and $Y$ respectively. If
$d_Y(f(x),f(y)) = d_X(x,y)$ for all $x, y \in X$ then $f$ is called an
\textit{isometric embedding}; if in addition $f$ is surjective then $f$ is
an isometry. More generally, let 
$1 \leq \lambda < \infty$, $0 \leq \mu < \infty$ and $0 < \epsilon < \infty$ be constants.
The map $f$ is called a \emph{$(\lambda,\epsilon)$-quasi-isometric embedding},
and $X$ \textit{embeds quasi-isometrically in $Y$}, if
\[
\frac{1}{\lambda}\;d_X(x,y) - \epsilon \leq
d_Y(f(x), f(y)) \leq \lambda d_X(x,y) + \epsilon
\]
for all $x,y \in X$.

A subset $Z \subseteq Y$ is called \textit{$\mu$-quasi-dense} if for every
$y \in Y$ there exists a $z \in Z$ with $d_Y(y,z) \leq \mu$ and $d_Y(z,y) \leq \mu$. 
If $f : X \to Y$ is a $(\lambda,\epsilon)$-quasi-isometric embedding
and its image is $\mu$-quasi-dense, then $f$ is 
called a \textit{$(\lambda,\epsilon,\mu)$-quasi-isometry}, and the spaces
$X$ and $Y$ are said to be \textit{quasi-isometric}. Quasi-isometry forms
an equivalence relation on the class of semimetric spaces
\cite[Proposition~1]{K_semimetric}.
A semimetric space is called \textit{quasi-metric} if it is quasi-isometric
to a metric space, or equivalently \cite[Proposition~2]{K_semimetric} if there are
constants $\lambda, \mu < \infty$ such that $d(x,y) \leq \lambda d(y,x) + \mu$
for all points $x$ and $y$.

Now let $M$ be a monoid generated by a finite subset $S$. Then $M$ is naturally
endowed with the structure of a directed graph, with vertices the elements
of $M$, and an edge from $x$ to $y$ if and only if there is a generator
$s \in S$ such that $xs = y$ in $M$. This graph is called the \textit{(right)
Cayley graph of $M$ with respect to the generating set $S$}. The Cayley graph
in turn has the
structure of a semimetric space, as described above, with $d_S(x,y)$ being
the shortest length of a word $w$ over the generating set $S$ such that
$xw = y$ in $M$, or $\infty$ if there is no such word.

Of course different choices of finite generating set for $M$ will lead
to different graphs and different semimetric spaces, but
two different finite generating sets for the same monoid
will always give rise to quasi-isometric spaces \cite[Proposition~4]{K_semimetric}. In other words,
provided a monoid admits a finite generating set, its quasi-isometry class
is an invariant, and so it makes sense to speak of two abstract 
finitely generated
monoids
being quasi-isometric.

Given two functions $f, g : \mathbb{N}\to \mathbb{N}$ we write $f \prec g$ if there exists
a constant $a$ such that $f(j) \leq ag(aj) + aj$ for all $j$. The functions
$f$ and $g$ are said to be of the same \textit{type}, written $f \sim g$, if $f \prec g$ and
$g \prec f$.

Now fix a monoid presentation $\langle A \mid R \rangle$.
If $u$ and $v$ are equivalent words then the \textit{area} $A(u,v)$ is
the smallest number of applications of relations from $R$ necessary to
transform $u$ into $v$.
The \textit{Dehn function} of a presentation $\langle A \mid R \rangle$ is
the function
$\delta : \mathbb{N} \to \mathbb{N}$ given by
$$\delta(n) = \max \lbrace A(u,v) \mid u,v \in A^*, u \equiv_R v, |u|+|v| \leq n \rbrace.$$
The Dehn function is a measure of the complexity of transformations between
equivalent words. 
The Dehn function depends on the presentation, but if $\delta$ and $\gamma$
are Dehn functions of different finite presentations for the same monoid
then $\delta \sim \gamma$.

\section{Geometric Nature of Finite Presentability}\label{sec_fp}

In this section we describe an elementary, but very useful, geometric
property which, when applied to Cayley graphs, characterises finite
presentability for left cancellative monoids. We then show that this
property is invariant under quasi-isometry, from which follows the result
that finite presentability is a quasi-isometry invariant of finitely
generated left cancellative monoids.

We shall need
the notion of a \textit{directed $2$-complex}, which was introduced by
Guba and Sapir \cite{Guba06}.
For every directed graph $\Gamma$ let $P(\Gamma)$ be the set of all directed
paths in $\Gamma$, including the empty paths. We write $\iota p$ and $\tau p$
for the start and end vertex respectively of a path $p$. A pair of paths
$p,q \in P(\Gamma)$ are said to be \textit{parallel}, written
$p \parallel q$, if $\iota p = \iota q$ and $\tau p = \tau q$. 

A {\em directed $2$-complex}
is a directed graph $\Gamma$ equipped with a set $F$ (called the {\em set
of $2$-cells}), and three maps $\topp{\cdot}\colon F \to P$,
$\bott{\cdot}\colon F \to P$, and $^{-1}\colon\to F$ called {\em top},
{\em bottom}, and {\em inverse} such that
\begin{itemize}
\item for every $f\in F$, the paths $\topp{f}$ and $\bott{f}$ are parallel;
\item $^{-1}$ is an involution without fixed points, and
$\topp{f^{-1}}=\bott{f}$, $\bott{f^{-1}}=\topp{f}$ for every $f\in F$.
\end{itemize}

If $K$ is a directed $2$-complex, then paths on $K$ are called
{\em $1$-paths}. The initial and
terminal vertex of a $1$-path $p$ are denoted by $\iota(p)$ and
$\tau(p)$, respectively. For every $2$-cell $f\in F$, the vertices
$\iota(\topp{f})=\iota(\bott{f})$ and $\tau(\topp{f})=\tau(\bott{f})$ are
denoted $\iota(f)$ and $\tau(f)$, respectively.

An \emph{atomic $2$-path} is a triple $(p,f,q)$, where $p$, $q$
are $1$-paths in $K$, and $f \in F$ such that $\tau(p) = \iota(f)$,
$\tau(f) = \iota(q)$. If $\delta$ is an atomic $2$-path then we use $\topp{\delta}$ to denote $p \topp{f} q$ and $\bott{\delta}$ is denoted by $p \bott{f} q$, these are the top and bottom $1$-paths of the atomic $2$-path. A
$2$-path $\delta$ in $K$ of length $n$ is then a sequence of atomic paths
$\delta_1$, $\ldots$, $\delta_n$, where $\bott{\delta_i} = \topp{\delta_{i+1}}$ for every $1 \leq i < n$. The top and bottom $1$-paths of $\delta$, denoted $\topp{\delta}$ and $\bott{\delta}$ are then defined as $\topp{\delta_1}$ and $\bott{\delta_n}$, respectively. 

We use $\delta \circ \delta'$ to denote the composition of two $2$-paths.
We say that $1$-paths $p$, $q$ in $K$ are \textit{homotopic} if there exists a
$2$-path $\delta$ such that $\topp{\delta} = p$ and $\bott{\delta} = q$. 
We say that a directed $2$-complex $K$ is \emph{directed simply connected}
if for every pair of parallel paths $p \parallel q$, $p$ and $q$ are
homotopic in $K$. 

Let $K$ be a directed $2$-complex with underlying directed graph $\Gamma$
and set of $2$-cells $F$. Let $p$ and $q$ be parallel paths in $\Gamma$,
and let $K'$ be the $2$-complex obtained from $K$ by
adjoining two new elements $f$ and $f'$ to $F$ satisfying $\topp{f} = \bott{f'} = p$,
$\bott{f} = \topp{f'} = q$. We call $K'$ the directed $2$-complex
\textit{obtained from $K$ by adjoining cells for the paths $p$ and $q$}.  

Given a directed graph $\Gamma$ and natural number $n$ we define a directed
$2$-complex $K_n(\Gamma)$ with face set
$$F = \lbrace (p, q) \mid p \textrm{ and } q \textrm{ are parallel paths in } \Gamma \textrm{ with } |p| + |q| \leq n \rbrace$$
and $\topp{(p,q)} = p$, $\bott{(p,q)} = q$ and $(p,q)^{-1} = (q,p)$.
For $n \in \mathbb{N}$, we say that a directed graph $\Gamma$ is
\textit{$n$-quasi-simply-connected} if $K_n(\Gamma)$ is directed
simply connected.  We say that $\Gamma$ is \textit{quasi-simply-connected}
if it is $n$-quasi-simply-connected for some $n \in \mathbb{N}$.

The directed $2$-complex $K_n(\Gamma)$ is the natural directed analogue
of the \emph{Rips complex}, and
Theorem~\ref{thm_LeftCanFPGeom} below is the cancellative monoid analogue
of the well-known result in geometric group theory which states that a
group $G$ with generating set $A$ is finitely presented if and only if
the Rips complex $Rips_r(G,A)$ is simply connected for $r$ large enough;
see \cite[Chapter~4]{delaHarpe00}. 

Let $K_n(\Gamma)$ be a directed simply connected $2$-complex.
For each pair of parallel paths $p \parallel q$ in $\Gamma$ define the
area $A_{K_n(\Gamma)}(p,q)$ to be the minimum length of a $2$-path from
$p$ to $q$ in $K_n(\Gamma)$. The Dehn function
$\gamma: \mathbb{N} \rightarrow \mathbb{Z}^+ \cup \{ \infty \}$ of $K_n(\Gamma)$ is defined by
\[
\gamma(i) = \sup\{ A_{K_n(\Gamma)}(p,q) \ \mbox{in $K_n(\Gamma)$}: p,q \in \Gamma, p \parallel q, |p|+|q| \leq i \}, 
\]  
where the supremum of an unbounded set is taken to be $\infty$.

\begin{theorem}
\label{thm_LeftCanFPGeom}
Let $S$ be a left cancellative monoid generated by a finite set $A$. Then
$S$ is finitely presented if and only if the right Cayley graph
$\Gamma_r(S,A)$ is quasi-simply-connected. Moreover, if $\Gamma_r(S,A)$ is
$n$-quasi-simply-connected then $K_n(\Gamma_r(S,A))$ has Dehn function
equivalent to the Dehn function of $S$.
\end{theorem} 
\begin{proof}
First suppose that $S$ is presented by a finite presentation
$\langle A \mid R \rangle$, with Dehn function $\delta : \mathbb{N} \to \mathbb{N}$. Then it is not hard to see that the right
Cayley graph $\Gamma = \Gamma_r(S,A)$ is quasi-simply-connected with 
\[
n = \max \{ |u| + |v| : (u=v) \in R \}. 
\]   
Indeed, let $p, q \in P(\Gamma)$ with $p \parallel q$ and let $w_p$ and
$w_q$ be the words labelling the paths $p$ and $q$ respectively. Since
$S$ is left cancellative, $w_p=w_q$ in $S$, and so there is a
finite sequence of applications of relations from $R$ that transforms
$w_p$ into
$w_q$. Moreover, this sequence can be chosen to have length
at most $\delta(|w_p| + |w_q|) = \delta(|p| + |q|)$. This sequence gives rise in
a natural way to a $2$-path, of the same length, in $K_n(\Gamma)$ from
$p$ to $q$. Thus, $K_n(\Gamma)$ is directed simply connected with Dehn
function bounded above by $\delta$.

Conversely, suppose we are given that $\Gamma = \Gamma_r(S,A)$ is
$n$-quasi-simply-connected. Suppose $K_n(\Gamma)$ has
Dehn function bounded above by $\omega : \mathbb{N} \to \mathbb{N}$. Let $R$ be the set of all
relations $u=v$ over $A$
holding in $S$ with $|u| + |v| \leq n$. Since $A$ is finite, $R$ is finite.
We claim that $\langle A \mid R \rangle$ defines the monoid $S$. By
definition all of these relations hold in $S$, so we need only show this
set of relations is sufficient to define $S$. Given $\alpha, \beta \in A^*$
such that $\alpha = \beta$ in $S$, let $p_{\alpha}$, $p_{\beta}$ be the
paths in $\Gamma$ labelled by $\alpha$, $\beta$ respectively
and with $\iota p_{\alpha} = \iota p_{\beta} = 1_S$ and
$\tau p_{\alpha} = \tau p_{\beta} = \alpha = \beta$. By assumption
$K_n(\Gamma)$ is directed simply connected, so there is a $2$-path
from $p_\alpha$ to $p_\beta$ in $K_n(\Gamma)$, of length at most
$\omega(|p_\alpha| + |p_\beta|) = \omega(|\alpha| + |\beta|$).

Now for any face $f$ in this $2$-path, $\bott{f}$ and $\topp{f}$ are
parallel paths in $\Gamma$, which since $S$ is left cancellative means
that their labels represent the same element of $S$. Moreover, since
$f$ is a face in $K_n(\Gamma)$, their labels have total length less than
$n$, and hence form the two sides of a relation in $R$. It follows that the
$2$-path corresponds to a sequence of applications of relations from $R$
which transforms the word $\alpha$ into the word $\beta$.

Moreover, this sequence has length at most $\omega(|\alpha| + |\beta|)$.
Thus, $S$ is finitely presented with Dehn function bounded above by $\omega$.
Finally, since $S$ is finitely presented, we may now apply the first part of
the proof again to deduce that $\omega$ is bounded above by the Dehn function
for the presentation, which means that the two Dehn functions are equal.
\end{proof}

It is natural to ask to what extent the left cancellativity assumption in the above theorem really is necessary. As it turns out, for finitely generated monoids in general being quasi-simply connected is neither a
necessary nor a sufficient condition for the existence of a finite presentation. 
In Section~\ref{sec_examples} below, we shall see an example of a finitely
generated monoid with Cayley graph which is a directed tree, but which is
not finitely presented. This shows that being quasi-simply-connected is not
sufficient to imply a finite presentation in general. Also in
Section~\ref{sec_examples} we shall construct an example of a
finitely presented monoid whose Cayley graph is not quasi-simply connected.

We shall now show that quasi-simply-connectedness is a quasi-isometry
invariant of directed graphs, from which it will follow that finite
presentability is a quasi-isometry invariant of finitely generated
left cancellaive monoids. 

The following general lemma will prove useful for us. 

\begin{lemma}[Quasi-inverses]
\label{lem_quasiinverses}
Let $X$ and $Y$ be quasi-isometric semimetric spaces. 
Then there exist constants $\lambda$, $\epsilon$ and $\mu$ and a pair
of $(\lambda,\epsilon,\mu)$-quasi-isometries $f:X \rightarrow Y$ and $g:Y \rightarrow X$ satisfying the following properties:
\begin{enumerate}
\item[(i)] $d(y,fg(y)) \leq \mu$ and $d(fg(y),y) \leq \mu$ for all $y \in Y$;
\item[(ii)] $d(x,gf(x)) \leq \mu$ and $d(gf(x),x) \leq \mu$ for all $x \in X$;
\item[(iii)] $gfg(y)=g(y)$ for all $y \in Y$;
\item[(iv)] $fgf(x)=f(x)$ for all $x \in X$. 
\end{enumerate}
\end{lemma}
\begin{proof}
Let $g : Y \rightarrow X$ be a $(\lambda',\epsilon',\mu')$-quasi-isometry. 
For every point $x \in X$ choose and fix $\hat{x} \in \im(g)$ satisfying
$d(x,\hat{x}) \leq \mu'$ and $d(\hat{x},x) \leq \mu'$. Now define a map
$f : X \rightarrow Y$ by choosing for each $z \in \im(g)$ a point
$f(z)$ such that $g(f(z)) = z$, and then extend to the whole of $X$ by
setting $f(x) = f(\hat{x})$ for all $x \in X$.

Then straightforward calculations show that for all $a,b \in X$ we have 
\[
d(f(a),f(b)) \geq \frac{1}{\lambda'}d(a,b) - \frac{(\epsilon' + 2\mu')}{\lambda'}, 
\]
and
\[
d(f(a),f(b)) \leq \lambda' d(a,b) + \lambda'(\epsilon' + 2\mu')
\]
so that $f$ is a quasi-isometric embedding.
Also, for all $y \in Y$,  
\[
d(y,f(g(y)) \leq \epsilon' 
\quad \mbox{and} \quad 
d(f(g(y),y) \leq \epsilon' 
\]
therefore $\im f$ is quasi-dense and $f$ is a 
$(\lambda', \sigma, \epsilon' + 1)$-quasi-isometry where 
\[
\sigma = \max\left(
\frac{(\epsilon' + 2\mu')}{\lambda'},
\lambda'(\epsilon' + 2\mu')
\right).
\]
Moreover, for all $x \in X$ and $y \in Y$ we have
\[
d(x,gf(x)) \leq \mu', \quad 
d(gf(x),x) \leq \mu',
\]
\[
gfg(y) = g(y) 
\ \mbox{and} \ 
fgf(x) = f(x). 
\]
It follows that $f$ and $g$ are $(\lambda,\epsilon,\mu)$-quasi-isometries satisfying the conditions given in the statement of the lemma where 
\[
\lambda = \lambda',
\quad
\epsilon = \max(\epsilon',\sigma),
\quad \mbox{and} \quad
\mu = \max(\mu',\epsilon'+1), 
\]
as required. 
\end{proof}

\begin{lemma}
\label{lem_digraphs}
Let $\Gamma$ and $\Delta$ be simple directed graphs and let $f:\Gamma \rightarrow \Delta$ and $g:\Delta \rightarrow \Gamma$ be $(\lambda,\epsilon,\mu)$-quasi-isometries satisfying (i)-(iv) from Lemma~\ref{lem_quasiinverses}. 
Suppose $K_n(\Gamma)$ is directed simply connected. Then $K_m(\Delta)$ is directly simply connected where
$m = \mathrm{max}(\lambda^2 + (\lambda + 1)\epsilon + 2\mu + 1, (\lambda + \epsilon)n )$. 

If, moreover, $K_n(\Gamma)$ and $K_m(\Delta)$ have Dehn functions $\gamma$
and $\delta$ respectively, then $\delta \sim \gamma$.
\end{lemma}
\begin{proof}
For each arc $\alpha$ from $a$ to $b$ in $\Gamma$, where $a$ and $b$ are vertices, choose and
fix a geodesic path $\pi(\alpha)$ in $\Delta$ from $f(a)$ to $f(b)$; note that
$d(f(a),f(b)) \leq \lambda + \epsilon$, so such a geodesic exists. The map
$\pi$ extends naturally to a map from $P(\Gamma)$ to $P(\Delta)$ which we
also denote by $\pi$. Let $F$ and $E$ be the sets of $2$-cells of
$K_n(\Gamma)$ and $K_m(\Delta)$ respectively. By definition of $m$
(considering the right hand term) for all $f \in F$ there exists
$e \in E$ such that $\topp{e} = \pi(\topp{f})$, $\bott{e} = \pi(\bott{f})$.
For each $f \in F$ choose and fix such an $e \in E$ and define $\pi(f) = e$.
For each atomic $2$-path $(p,f,q)$ of $K_n(\Gamma)$, define $\pi(p,f,q) = (\pi(p),\pi(f),\pi(q))$.
Clearly this is an atomic $2$-path of $K_m(\Delta)$. This now extends in an obvious way
to a
mapping $\pi$ from $2$-paths of $K_n(\Gamma)$ to $2$-paths of $K_m(\Delta)$.
In other words, $\pi:K_n(\Gamma) \rightarrow K_m(\Delta)$ is a morphism of
directed $2$-complexes (in the sense of \cite[Section~5]{Guba06}).

Now let $p,q \in P(\Gamma)$ with $p \parallel q$. By assumption there is a
$2$-path in $K_n(\Gamma)$ from $p$ to $q$. The image of this $2$-path under
$\pi$ is then a $2$-path in $K_m(\Delta)$ from $\pi(p)$ to $\pi(q)$, of the
same length.

Next suppose $r$ and $s$ are parallel
paths in $\Delta$. Suppose the vertices of these paths, in order, are
$r_0,r_1, \dots, r_c$ and $s_0,s_1, \dots, s_d$ respectively. Then
$r_0 = s_0$, $r_c=s_d$. For each $i$ let $\sigma_i$ be a path in $\Delta$
from $r_i$ to $fg(r_i)$ with $|\sigma_i| \leq \mu$ and let $\sigma_i^{-1}$
be a path in $\Delta$ from $fg(r_i)$ to $r_i$ with $|\sigma_i^{-1}| \leq \mu$. 
For $0 \leq i < c$, choose a geodesic $\tau_i$ in $\Gamma$ from $g(r_i)$ to
$g(r_{i+1})$. Similarly, for each $0 \leq j < d$, choose a geodesic $\zeta_j$
in $\Gamma$ from $g(s_j)$ to $g(s_{j+1})$.

Let $\tau = \tau_0 \dots \tau_{c-1}$
and $\zeta = \zeta_0 \dots \zeta_{d-1}$. Since $g$ is a
$(\lambda, \epsilon,\mu)$-quasi-isometry and we have 
$|\tau_i| = d(g(r_i), g(s_{i+1})) \leq \lambda + \epsilon$, and hence
$|\tau| \leq c (\lambda+\epsilon)$. Similarly,
$|\zeta| \leq d (\lambda+\epsilon)$.
Now $\tau$ and $\zeta$ are parallel paths of length in $\Gamma$ and hence
in $K_n(\Gamma)$. Since $K_n(\Gamma)$ is directed simply connected, this
means there is a $2$-path $\phi$ from $\tau$ to $\zeta$. And since $\delta$ is
the Dehn function of $K_n(\Gamma)$, we may choose $\phi$ of length at most
$\delta(|\tau| + |\zeta|) \leq \delta((c+d)(\lambda + \epsilon))$.

Now by the above observations, $\pi(\phi)$ is a $2$-path from $\pi(\tau)$
to $\pi(\zeta)$, of length at most $\delta((c+d)(\lambda+\epsilon))$.

Moreover, by definition of $m$ (considering the left hand term) there exists
$e_0 \in E$ with 
\[
\topp{e_0} = \sigma_0 \pi(\tau_0) \sigma_1^{-1},
\quad
\bott{e_0} = (r_0,r_1), 
\]
and for $1 \leq i \leq c-1$, there exist $e_i \in E$ such that
\[
\topp{e_i} = \pi(\tau_i) \circ \sigma_{i+1}^{-1}, 
\quad
\bott{e_i} = \sigma_i^{-1} \circ (r_i,r_{i+1}). 
\]
These combine, as illustrated in Figure~\ref{fig_2path}, 
to give an atomic
$2$-path of length $c$ from $r$ to $\sigma_0 \pi(\tau) \sigma_c^{-1}$.
An entirely similar argument shows that there is a $2$-path of length
$d$ from $\sigma_0 \pi(\zeta) \sigma_c^{-1}$ to $s$, and we have already seen that there is a $2$-path
of length at most $\delta((c+d)(\lambda+\epsilon))$ from $\pi(\tau)$ to
$\pi(\zeta)$, and hence there is a $2$-path of the same length from $\sigma_0 \pi(\tau) \sigma_c^{-1}$ to
$\sigma_0 \pi(\zeta) \sigma_c^{-1}$. Thus, there is a $2$-path of length at most
$$c+d + \delta((c+d)((\lambda + \epsilon))$$
where $c+d = |r|+|s|$. This shows that $K_m(\Delta)$ is directed simply connected
and $\delta \prec \gamma$ as required. Moreover, now we know that $K_m(\Delta)$
is directed simply connected, we may apply what we have proved with
$\Gamma$ and $\Delta$, to yield $\gamma \prec \delta$ and hence
$\gamma \sim \delta$.
\end{proof}

\begin{figure}
\begin{center}
\begin{tikzpicture}
[decoration={ 
markings,
mark=
at position 0.56 
with 
{ 
\arrow[scale=1.5]{stealth} 
} 
} 
] 
\tikzstyle{vertex}=[circle,draw=black, fill=white, inner sep = 0.4mm]

\matrix[row sep=5mm,column sep = 2cm]{

\node (v0) [vertex,label={180:{\tiny $fg(s_0)=fg(r_0)$}}] at (180:4cm) {};
\node (v1) [vertex,label={135:{}}] at (150:4cm) {};
\node (v2) [vertex,label={90:{}}] at (120:4cm) {};
\node (v3) [vertex,label={135:{}}] at (90:4cm) {};
\node (v4) [vertex,label={90:{}}] at (60:4cm) {};
\node (v5) [vertex,label={90:{}}] at (30:4cm) {};
\node (v6) [vertex,label={00:{\tiny $fg(r_c)=fg(s_d)$}}] at (0:4cm) {};
\draw [postaction={decorate}] (v0)--(v1);
\draw [postaction={decorate}] (v1)--(v2);
\draw [postaction={decorate}] (v2)--(v3);
\draw [postaction={decorate}] (v3)--(v4);
\draw [postaction={decorate}] (v4)--(v5);
\draw [postaction={decorate}] (v5)--(v6);
\node (w0) [vertex,label={00:{\tiny $r_0=s_0$}}] at (180:2cm) {};
\node (w1) [vertex,label={-30:{\tiny $r_1$}}] at (150:2cm) {};
\node (w2) [vertex,label={-60:{\tiny $r_2$}}] at (120:2cm) {};
\node (w3) [vertex,label={-90:{}}] at (90:2cm) {};
\node at (90:1.5cm) {\tiny $\ldots$};
\node at (-90:1.5cm) {\tiny $\ldots$};
\node at (90:1.2cm) {\tiny $r$};
\node at (-90:1.2cm) {\tiny $s$};
\node at (180-15:4.3cm) {\tiny $\pi(\tau_0)$};
\node at (180-15-30:4.3cm) {\tiny $\pi(\tau_1)$};
\node at (15:4.5cm) {\tiny $\pi(\tau_{c-1})$};
\node at (90:4.8cm) {\tiny $\pi(\tau)$};
\node at (-90:4.8cm) {\tiny $\pi(\zeta)$};
\node (w4) [vertex,label={90:{}}] at (60:2cm) {};
\node (w5) [vertex,label={90:{}}] at (30:2cm) {};
\node (w6) [vertex,label={180:{\tiny $r_c=s_d$}}] at (0:2cm) {};
\draw [postaction={decorate}] (w0)--(w1);
\draw [postaction={decorate}] (w1)--(w2);
\draw [postaction={decorate}] (w2)--(w3);
\draw [postaction={decorate}] (w3)--(w4);
\draw [postaction={decorate}] (w4)--(w5);
\draw [postaction={decorate}] (w5)--(w6);
\draw [postaction={decorate}] (w0)--(v0);
\draw [postaction={decorate}] (v6)--(w6);
%
%
%
%
%
%
%
\node (v0') [vertex,label={90:{}}] at (-180:4cm) {};
\node (v1') [vertex,label={135:{}}] at (-150:4cm) {};
\node (v2') [vertex,label={90:{}}] at (-120:4cm) {};
\node (v3') [vertex,label={135:{}}] at (-90:4cm) {};
\node (v4') [vertex,label={90:{}}] at (-60:4cm) {};
\node (v5') [vertex,label={90:{}}] at (-30:4cm) {};
\node (v6') [vertex,label={90:{}}] at (-0:4cm) {};
\draw [postaction={decorate}] (v0')--(v1');
\draw [postaction={decorate}] (v1')--(v2');
\draw [postaction={decorate}] (v2')--(v3');
\draw [postaction={decorate}] (v3')--(v4');
\draw [postaction={decorate}] (v4')--(v5');
\draw [postaction={decorate}] (v5')--(v6');
\node (w0') [vertex,label={90:{}}] at (-180:2cm) {};
\node (w1') [vertex,label={30:{\tiny $s_1$}}] at (-150:2cm) {};
\node (w2') [vertex,label={60:{\tiny $s_2$}}] at (-120:2cm) {};
\node (w3') [vertex,label={135:{}}] at (-90:2cm) {};
\node (w4') [vertex,label={90:{}}] at (-60:2cm) {};
\node (w5') [vertex,label={90:{}}] at (-30:2cm) {};
\node (w6') [vertex,label={90:{}}] at (-0:2cm) {};
\draw [postaction={decorate}] (w0')--(w1');
\draw [postaction={decorate}] (w1')--(w2');
\draw [postaction={decorate}] (w2')--(w3');
\draw [postaction={decorate}] (w3')--(w4');
\draw [postaction={decorate}] (w4')--(w5');
\draw [postaction={decorate}] (w5')--(w6');
%
%
%
%
%
\draw [postaction={decorate}] (v1)--(w1);
\node at (144:3cm) {\tiny $\sigma_1^{-1}$};
\node at (176:3cm) {\tiny $\sigma_0$};
\node at (83:3cm) {\tiny $\ldots$};
%
%
\draw [postaction={decorate}] (v2)--(w2);
\node at (113:3cm) {\tiny $\sigma_2^{-1}$};
\node at (06:3cm) {\tiny $\sigma_c^{-1}$};
%
\draw [postaction={decorate}] (v3)--(w3);
%
\draw [postaction={decorate}] (v4)--(w4);
%
\draw [postaction={decorate}] (v5)--(w5);
%
%
%
%
%
%
%
\draw [postaction={decorate}] (v1')--(w1');
%
\draw [postaction={decorate}] (v2')--(w2');
%
\draw [postaction={decorate}] (v3')--(w3');
%
\draw [postaction={decorate}] (v4')--(w4');
%
\draw [postaction={decorate}] (v5')--(w5');
%
%
%
%
\\
\\
};

\end{tikzpicture}
\end{center}
\caption{
An illustration of the proof of Lemma~\ref{lem_digraphs}.
There is a $2$-path from $r$ to $\sigma_0 \pi(\tau) \sigma_c^{-1}$ given by first replacing $(r_0,r_1)$ by $\sigma_0 \pi(\tau_0) \sigma_1^{-1}$, then $\sigma_1^{-1} (r_1, r_2)$ by $\pi(\tau_1) \sigma_2^{-1}$, and so on. In a similar way one constructs a $2$-path from $s$ to $\sigma_0 \pi(\zeta) \sigma_c^{-1}$. 
}
\label{fig_2path}
\end{figure}
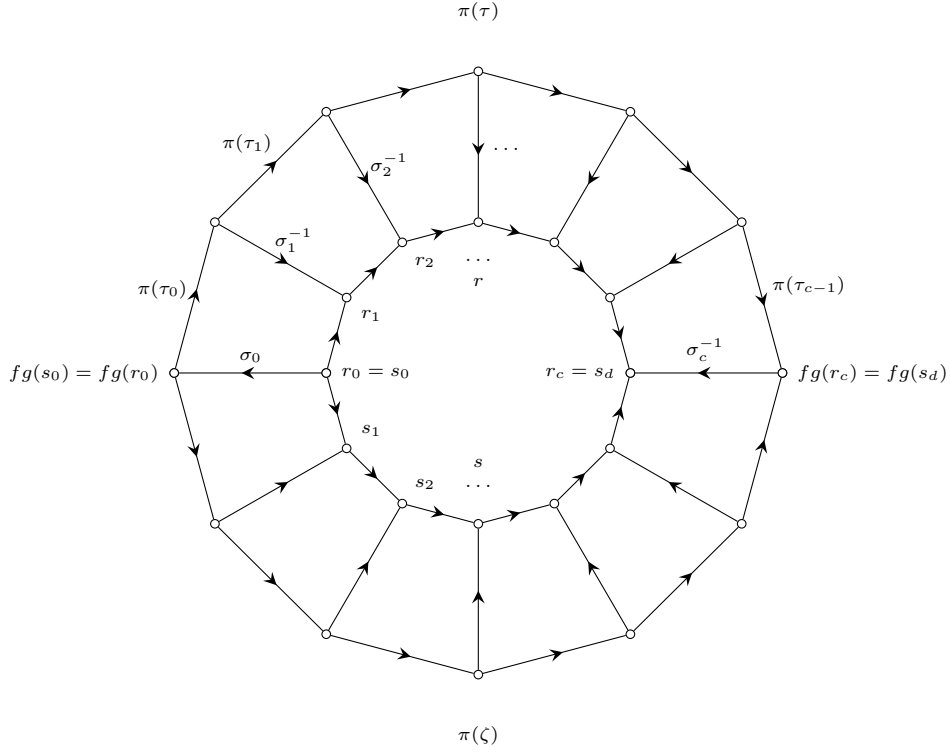

We are now ready to prove our main theorems.

\begin{theorema}
Let $M$ and $N$ be left cancellative, finitely generated monoids which are
quasi-isometric. Then $M$ is finitely presentable if and only if $N$ is
finitely presentable.
\end{theorema}
\begin{proof}
It follows from Lemmas~\ref{lem_quasiinverses} and \ref{lem_digraphs}
that the property of being quasi-simply-connected is a quasi-isometry invariant of
directed graphs. The result then follows by applying Theorem~\ref{thm_LeftCanFPGeom}. 
\end{proof}

\begin{theoremb}
Let $M$ and $N$ be left cancellative, finitely presentable monoids which
are quasi-isometric. Then $M$ has solvable word problem if and only if $N$
has solvable word problem.
\end{theoremb}
\begin{proof}
Let $\langle A \mid R \rangle$ and $\langle B \mid S \rangle$ be finite
presentations for $M$ and $N$ respectively, and let $\delta$ and $\gamma$
be the Dehn functions of these presentations respectively. Then by
Theorem~\ref{thm_LeftCanFPGeom}, there is an $n$ such that
$K_n(\Gamma_r(A,R))$ and $K_n(\Gamma_r(B,S))$ are directed simply connected
and moreover, if we let $\delta'$ and $\gamma'$ be the Dehn functions of
these graphs, then $\delta \sim \delta'$ and $\gamma \sim \gamma'$. Now
by Lemmas~\ref{lem_quasiinverses} and \ref{lem_digraphs} we have
$$\gamma \sim \gamma' \prec \delta' \sim \delta.$$
Since $M$ has solvable
word problem, $\delta$ is bounded above by a recursive function. Hence
so is $\gamma$, and so $N$ has solvable word problem.
\end{proof}

\section{The Non-Cancellative Case}\label{sec_examples}

In this section we present some examples showing that the theory developed above is very far from being extendable to arbitrary finitely generated monoids. 
We begin by giving examples which show that, 
for finitely generated monoids in general, being quasi-simply-connected
is neither a necessary nor sufficient condition for the existence of a
finite presentation.

First, let us see how to construct an example of a finitely presented monoid
whose Cayley graph is not quasi-simply connected. This serves as an
instructive example of how intuitions from group theory can fail in a more
general setting. It would seem intuitively clear that if a monoid is finitely
presented then one should be able to use the relations to build 2-paths
between arbitrary pairs of parallel paths in the Cayley graph. Indeed this
is true for 2-paths whose origin is the identity element of the monoid, but
in general there are many more 2-paths than that in the Cayley graph, and
without left cancellativity the idea of filling in parallel paths with
relations loses sense since one has to ``trace back" to the identity of
the monoid in order to find two words that are equal before one can start
applying relations from the presentation. 

Before presenting the example we shall need to introduce a some basic
notions from the structure theory of semigroups.
Recall that Green's relations $\mathcal{L}$ and $\mathcal{R}$ are
defined on any semigroup $S$ by $a \mathcal{L} b$ [respectively,
$a \mathcal{R} b$] if either $a = b$ or there exist elements
$c, d \in S$ with $a = cb$ and $b = da$ [respectively, $a = bc$ and $b = ad$].
Green's relation $\mathcal{H}$ is defined by $a \mathcal{H} b$
if and only if $a \mathcal{L} b$ and $a \mathcal{R} b$. All three relations
are equivalence relations. Notice it is immediate from the definitions that
the $\mathcal{R}$-classes of a finitely generated monoid (that is, equivalence
classes of the relation $\mathcal{R}$) are exactly the strongly connected
components of the Cayley graph.

It is well known that the $\mathcal{H}$-class of any idempotent is a maximal
subgroup. The notion of \textit{Sch\"utzenberger group} gives a useful way to
associate a group to an $\mathcal{H}$-class \textbf{not} containing an
idempotent. Let $H$ be an $\mathcal{H}$-class of $S$, and let
$\mathrm{Stab}(H) = \{s \in S : sH = H \}$ denote the (left) stabilizer
of $H$ under the action of $S$. We define an equivalence $\sigma = \sigma(H)$
on the stabilizer by $(x,y) \in \sigma$ if and only if $xh = xy$ for all
$h \in H$. It is straightforward to verify that $\sigma$ is a congruence,
and that $\mathcal{G}(H) = \mathrm{Stab}(H)/\sigma$ is a group, called
the left Sch\"utzenberger group of $H$. One can also define the right
Sch\"utzenberger group of $H$ in the natural way, and it turns out that
the left and right Sch\"utzenberger groups are isomorphic to one another.
For information about the basic properties of Sch\"utzenberger groups we
refer the reader to 
\cite[Section~2.3]{Lallement79}. 

Let $R$ be an $\mathcal{R}$-class of $H$. The (right) Sch\"utzenberger graph
$\Gamma(R, A)$ of $R$, with respect to $A$, is the strongly connected
component of $h \in H$ in $\Gamma(M,A)$. It is easily seen to consist
of those vertices which are elements of $R$, together with edges
connecting them, and so can be obtained by beginning with a directed
graph with vertex set $R$ and a directed labelled edge from $x$ to $y$ labelled by $a \in A$ if and only if $xa = y$. From its construction it is clear that for any generating set $A$ of $M$, $\Gamma(R,A)$ may be viewed a connected geodesic semimetric space.

The following observation results from the fact that the property of
being quasi-simply-connected is inherited by the strongly connected components of a directed graph (by which we mean the subdigraphs induced on the equivalence classes of vertices where two vertices $u$ an $v$ are in the same class if there is a directed path from $u$ to $v$ and a directed path back from $v$ to $u$). 

\begin{proposition}
\label{prop_strongcomp}
Let $S$ be a monoid generated by a finite set $A$ and let $R$ be an
$\gr$-class of $S$ with Sch\"utzenberger graph $\Delta(R)$. If
$\Gamma_r(S,A)$ is $n$-quasi-simply-connected then $\Delta(R)$ is
$n$-quasi-simply-connected. 
\end{proposition}
\begin{proof}
By definition $\Delta(R)$ is a strongly connected component of the
digraph $\Gamma_r(S,A)$. Suppose $p$ and $q$ are parallel paths in
$\Delta(R) \subseteq \Gamma_r(S,A)$. Since $\Gamma_r(S,A)$ is
$n$-quasi-simply connected there is a 2-path $\delta$ in
$K_n(\Gamma_r(S,A))$ from $p$ to $q$. Since $\Delta(R)$ is a strongly
connected component of $\Gamma_r(S,A)$, all $1$-paths featuring in $\delta$
lie inside $\Delta(R)$ and hence in $K_n(\Delta(R))$. But now by
the definition of $K_n(\Delta(R))$, all faces featuring in $\delta$ lie
in $K_n(\Delta)$, and so $\delta$ is also 
a $2$-path in 
$K_n(\Delta(R))$ from $p$ to $q$.
\end{proof}

\begin{corollary}\label{corol_passing}
Let $S$ be a monoid generated by a finite set $A$ and let $H$ be an $\gh$-class of $S$. If $S$ is quasi-simply connected and the $\gr$-class $R$ of $H$ contains only finitely many $\gh$-classes, then the Sch\"utzenberger group $\mathcal{G}(H)$ is finitely presented. 
\end{corollary}
\begin{proof}
The Sch\"utzenberger group $\mathcal{G}(H)$ acts naturally on the
Sch\"utzenberger graph $\Delta(R)$ in such a way that applying
the \v{S}varc-Milnor lemma for groups acting on semimetric spaces
\cite[Theorem~1]{K_semimetric} it follows that $\mathcal{G}(H)$ is a
finitely generated group which is quasi-isometric to $\Delta(R)$; see \cite[Section~5]{K_semimetric}. By Proposition~\ref{prop_strongcomp}, $\Delta(R)$ is quasi-simply connected. Since they are quasi-isometric, the group $\mathcal{G}(H)$ is quasi-simply connected by Lemma~\ref{lem_digraphs} which, by Theorem~\ref{thm_LeftCanFPGeom}, implies that the group $\mathcal{G}(H)$ is finitely presented. 
\end{proof}

\begin{proposition}
There exists a finitely presented monoid $S$ whose Cayley graph is not quasi-simply connected. 
\end{proposition}
\begin{proof}
Let $A$ be the alphabet 
\[
\{
a_1, a_2, a_3, a_4, a_1', a_2', a_3', a_4', b, c, d
\}
\]
and consider the presentation
\[
\langle A  \ |  \ a_j a_j' = a_j' a_j = \epsilon, \ a_1 a_2 = a_3 a_4, \ a_j b = b a_j^2, \ cb^2 = cb, \ a_j d = d a_j, 
\]
\[
cbda_j = a_jcbd  \ (j = 1,2,3,4) \rangle. 
\]
Let $S$ be the monoid defined by this presentation, and let $H$ be the $\gh$-class of $h \equiv cbd$. 
In \cite{Ruskuc00} using Reidemeister-Schreier rewriting methods it is shown that the Sch\"utzenberger group $\mathcal{G}(H)$ is defined by the following group presentation
\[
\langle
a_1, a_2, a_3, a_4 \ | \ a_1^{2^i} a_2^{2^i} = a_3^{2^i} a_4^{2^i} \ (i=0,1,2,\ldots) 
\rangle
\]
which is not finitely presented since it is an amalgamated product of
two free groups of rank two with a free group of infinite rank amalgamated.
It is also shown in \cite{Ruskuc00} that the $\gh$-class $H$ is the unique
$\gh$-class in its $\gr$-class. Therefore, by the contrapositive to
Corollary~\ref{corol_passing}, the finitely presented monoid $S$ is not
quasi-simply connected.
\end{proof}

Next, we exhibit an uncountable family of (non-left-cancellative)
finitely generated monoids which are pairwise non-isomorphic, but nevertheless
all share the same unlabelled Cayley graph. 
This Cayley graph will turn out to be isomorphic to a directed rooted tree with all vertices having out degree $4$ or $5$. 
We show that this family contains
examples of monoids with solvable word problem, 
and monoids with word problem
neither recursively enumerable nor co-recursively enumerable. This shows
that the  solvability of the word problem for general monoids is a
fundamentally ``non-geometric'' property which cannot be seen in a Cayley
graph.

We take the set $\mathbb{N}$ of natural numbers including $0$.
For each non-empty proper subset $X$ of $\mathbb{N}$, we define a
finitely generated monoid $M(X)$ by the following infinite presentation:
\[
\langle
a,b,c,d,e \ | \ 
a b^i c = a b^i d \ (i \in X), \quad a b^j c = a b^j e \ (j \notin X) 
\rangle.
\] 
We remark that since the monoids $M(X)$ are given by homogeneous presentations
with finitely many generators, they are residually finite. Indeed, any finite
set of elements can be seperated by a Rees quotient factoring out the ideal
consisting of all elements whose representatives have length $n$ or more,
for some sufficiently large $n$.

It is easy to show from the definition that the word problems for monoids
in this class can belong to a broad range of computability and complexity
classes:
\begin{proposition}\label{prop_complexity}
The word problem for $M(X)$ is linear-time Turing equivalent to the
membership problem for $X$ in unary coding. In particular, the word
problem for $M(X)$ is decidable if and only if $X$ is a recursive
subset of $\mathbb{N}$.
\end{proposition}
\begin{proof}
Given a solution to the word problem for $M(X)$, one may decide whether
$n \in \mathbb{N}$ by simply checking whether $a b^n c = a b^n d$ in $M(X)$.

Conversely, it is easy to see that the defining presentation for $M(X)$
yields an (infinite) terminating, convergent writing system:
$$a b^i c \to a b^i d \textrm{ for } i \in X, \quad a b^j c \to a b^j e \textrm{ for } j \notin X.$$
Clearly, it is an easy matter to check if a given word is left-hand-side of a
rule. Given an algorithm for membership of $X$, which can check of \textit{which}
rule a given word is the left-hand-side.
Thus, we can compute a normal
form for a word $u$ by iteratively checking its (finitely many) factors to
see if any is the left-hand-side of a rule, and if so applying the rule.
In fact, the complete lack of overlap between left-hand-sides and
right-hand-sides of rules means that this procedure can be performed by a single
pass from left to right across $u$, and the sum of all the values for
which membership of $X$ must be checked will not exceed the length of $u$. 
\end{proof}

\begin{proposition}
The word problem for $M(X)$ is recursively enumerable if and only if
it is co-recursively enumerable.
\end{proposition}
\begin{proof}
If the word problem for $M(X)$ is recursively enumerable, then by
the same argument as in the first part of the proof of
Proposition~\ref{prop_complexity}, so is $X$. Now notice that $M(X)$
is isomorphic to $M(\mathbb{N} \setminus X)$, via the map exchanging
$d$ and $e$. Since a recursively enumerable word problem is an isomorphism
invariant, the word problem for
 $M(\mathbb{N} \setminus X)$
  is recursively enumerable, and by the
same argument as above, so is $\mathbb{N} \setminus X$. Thus, $X$ is
recursive, and so by Proposition~\ref{prop_complexity}, $M(X)$ has
solvable word problem. A dual argument shows that $M(X)$ is
co-recursively enumerable if and only if it is recursive.
\end{proof}

\begin{proposition}\label{prop_trees}
For any subsets $X$ and $Y$ of $\mathbb{N}$, the semigroups $M(X)$ and
$M(Y)$ are isometric to each other, and to a directed rooted tree in
which every vertex has outdegree $4$ or $5$.
\end{proposition}
\begin{proof}
It follows from the confluence and termination of the rewriting system
in the proof of Proposition~\ref{prop_complexity} that
$$\mathcal{N} = A^* \setminus (A^* ab^ic A^*) \ (i \in \mathbb{N}), \ \ \ \ \textrm{ where } A = \{a,b,c,d,e\}$$
is a set of unique normal forms for the elements of $M(X)$. Note that this set
is independent of the choice of the $X$. Moreover, given a normal form $u$,
the normal forms for elements of the form $ux$ with $x$ a generator are:
\begin{itemize}
\item $ua, ub, ud, ue$ if $u = u' a b^i$ for some $u' \in A^*$ and $i \in \mathbb{N}$; or
\item $u a$, $ ub $, $uc$, $ud$ and $ue$, otherwise.
\end{itemize}
From this is it immediate that the Cayley graph of $M(X)$ is a rooted
directed tree in which every vertex has outdegree $4$ or $5$. Notice,
moreover, that the normal forms to which the normal form $u$ is connected
in the Cayley graph are independent of $X$. It follows that the identity
map on normal forms induces an isometry between $M(X)$ and $M(Y)$ for any
subsets of $X$ and $Y$ of $\mathbb{N}$.
\end{proof}

We can also use this example to show that finite presentability is not a
quasi-isometry invariant of finitely generated monoids considered with the
(symmetric) metric induced by its Cayley graph regarded as an undirected
graph. An immediate corollary of Proposition~\ref{prop_trees} is that the
\textbf{undirected} Cayley graphs of the monoids of the form $M(X)$ are
all isometric, and are all trees in which every vertex has degree $5$
or $6$.

It is a well-known, if at first a little surprising, fact in geometric
group theory that any free group of finite rank exceeding $2$ is 
quasi-isometric to the free group of rank $2$. This is actually a special
case of a more general phenomenon involving quasi-isometries between
locally finite trees. 

The simplest non-elementary Gromov hyperbolic metric spaces are homogeneous
simplicial trees $T$ of constant valency $\geq 3$. One interesting feature
of such
geometries is that all trees with constant
valency $\geq 3$ are quasi-isometric to each other.
Indeed, as observed in  \cite[Section~2.1]{Mosher03} and \cite{Alvaro12}, each such tree is quasi-isometric
to any tree $T$ satisfying the following properties: 
\begin{itemize}
\item $T$ has bounded
valency, meaning that vertices have uniformly finite valency; and
\item $T$ is \emph{bushy},
meaning that each point of $T$ is a uniformly bounded distance from a vertex
having at least $3$ unbounded complementary components. 
\end{itemize}
\begin{lemma}
\label{lem_busy}
Let $T_1$ and $T_2$ be locally finite graph theoretic trees. If $T_1$ and $T_2$ have bounded degree and every vertex in $T_1$ and in $T_2$ has degree at least $3$ then $T_1$ and $T_2$ are quasi-isometric.
\end{lemma}
\begin{proof}
This follows from the fact that $T_1$ and $T_2$ are both bushy trees. 
\end{proof}

\begin{theorem}
\label{thm_unqsi}
For every subset $X$ of the natural numbers, the undirected Cayley graph of the finitely generated monoid $M(X)$ (defined above) is quasi-isometric to the Cayley graph of the free group $F_2$. 

In particular finite presentability is not an undirected quasi-isometry invariant of finitely generated monoids. 
\end{theorem}
\begin{proof}
We observed above that the right directed Cayley graph $\Gamma(M(X))$ is a directed rooted tree in which every vertex has in degree $1$, and out degree either $4$ or $5$. It follows that the corresponding undirected Cayley graph is a tree in which every vertex has degree $5$ or $6$. Now the result follows by Lemma~\ref{lem_busy} since the undirected Cayley graphs of $M(X)$ and the free group $F_2$ are both bushy trees. 
\end{proof}

\bibliographystyle{plain}

\def\cprime{$'$} \def\cprime{$'$}

\end{document}